\documentclass[a4paper,10pt]{article}
\usepackage[english]{babel}
\usepackage[utf8]{inputenc}

\usepackage[nottoc]{tocbibind}
\usepackage{amsthm}
\usepackage{bm}
\usepackage{hyperref}
\usepackage{enumerate}
\usepackage{graphicx, psfrag, amscd, amssymb,amsmath,amsthm, amsfonts}
\usepackage{mathrsfs}
\usepackage[dvipsnames]{xcolor}
\usepackage{upgreek}

\newtheorem{theorem}{Theorem}[section]           
\newtheorem{corollary}[theorem]{Corollary}       
\newtheorem{lemma}[theorem]{Lemma}               
\newtheorem*{intro_theorem}{Theorem}

\theoremstyle{definition}   
\newtheorem{definition}[theorem]{Definition}     
\newtheorem*{intro_definition}{Definition}
\newtheorem{example}[theorem]{Example}           
\theoremstyle{remark}
\newtheorem{remark}[theorem]{Remark}             

\newcommand{\te}{\textrm}

\DeclareMathOperator{\spt}{spt}
\DeclareMathOperator{\im}{im}

\DeclareMathOperator{\domain}{dmn}
\DeclareMathOperator{\tangent}{Tan}

\DeclareMathOperator{\closure}{Clos}
\DeclareMathOperator{\interior}{Int}
\DeclareMathOperator{\boundary}{Bdry}

\DeclareMathOperator{\dist}{dist}
\DeclareMathOperator{\with}{:} 
\DeclareMathOperator{\without}{\sim}

\newcommand{\ud}{\ensuremath{\,\mathrm{d}}}

\newcommand{\restrict}{\mathop{\llcorner}}

\title{Integral Decompositions of Varifolds}
\author{Hsin-Chuang Chou}

\date{\today}

\begin{document}
\maketitle


\begin{abstract}
   This paper introduces a notion of decompositions of integral varifolds into countably many integral varifolds, and the existence of such decomposition of integral varifolds whose first variation is representable by integration is established. Furthermore, this result can be generalized by replacing the class of integral varifolds by some classes of rectifiable varifolds whose density is uniformly bounded from below. However, all of these decompositions may fail to be unique.
\end{abstract}

\section*{Introduction}

\textbf{General hypothesis.}
In this section, we suppose $m, n, Q$ are positive integers, $U$ is an open subset of $\mathbf{R}^n$, $V \in \mathbf{V}_m(U)$, and $\|\updelta V\|$ is a Radon measure.


\textbf{Motivation.}
To study the regularty of area-minimizing currents in higher codimension, Almgren introduced weakly differentiable $Q$-valued functions in \cite{Almgren2000_MR1777737}, which was extended by De Lellis and Spadaro in an intrinsic way in \cite{DLS11_MR2663735}. To study the theory of PDEs on varifolds, Menne introduced weakly differentiable functions on varifolds in \cite{Menne16a_MR3528825}; in this theory, a central element is the Poincar\'e inequality, a special case of which is the constancy theorem 
\cite[8.34]{Menne16a_MR3528825}: \textit{If the weak derivative vanishes, then the function must be constant on some decomposition of the varifold.} To study the convergence of pairs of varifolds and weakly differentiable functions thereon, or the theory of weakly differentiable multiple-valued functions on integral varifolds, we need a suitable notion of indecomposability of varifolds.


\textbf{Description of results.}
In \cite[6.2]{Menne16a_MR3528825}, the notion of indecomposability of $V$ was introduced by means of the distributional $V$ boundary of sets; more precisely, $V$ is termed \textit{indecomposable} if and only if there exists no $\|V\| + \|\updelta V\|$ measurable set $E$ such that $\|V\|(E) > 0$, $\|V\|(U \without E) > 0$, and the \textit{distributional $V$ boundary $V \partial E$ of $E$} is identically zero. In this case, a plane of multiplicity $2$ is indecomposable. 
On the other hand, since we expect our theory to coincide with the one of Almgren in case that $V$ is a plane of density $Q$, the Poincar\'e inequality \cite[2.12]{DLS11_MR2663735} suggests that $V$ should be decomposable into $Q$ identical planes of density $1$. This leads to the following definition:

\begin{intro_definition}[see Definition \ref{definition:integral indecomposability}]
    Suppose $V \in \mathbf{IV}_m(U)$. Then $V$ is called \textit{integrally indecomposable} if and only if there exists no $W \in \mathbf{IV}_m(U)$ such that $W \leq V$, $W \neq 0$, $V-W \neq 0$, $\|V\| = \|W\| + \|V-W\|$ and $\| \updelta V \| = \| \updelta W \| + \| \updelta (V-W) \|$.
\end{intro_definition}

Roughly speaking, we allow an integral varifold $V$ to be decomposed not only by restriction to subsets of zero distributional $V$ boundary but also by peeling off into several sheets without producing extra boundary.
For instance, if $V \in \mathbf{IV}_2(\mathbf{R}^2)$ satisfies $\|V\| = 2 \mathscr{L}^2$, then $\updelta V = 0$; letting $W = 2^{-1} V$, we see that $V$ is integrally decomposed as $V = W + W$.
It turns out that the two definitions of indecomposability are equivalent if $V \in \mathbf{IV}_m(U)$ has density $1$ for $\|V\|$ almost everywhere and $\|\updelta V\|$ is absolutely continuous with respect to $\|V\|$ (see \ref{lemma:equivalence if} and \ref{lemma:equivalence only if}). The aim of this paper is to show the following existence theorem of integral decomposition:

\begin{intro_theorem}[see Theorem \ref{theorem:integral decomposition}]
    Suppose $m, n \in \mathscr{P}$, $U$ is an open subset of $\mathbf{R}^n$, $V \in \mathbf{IV}_m(U)$, and $\|\updelta V\|$ is a Radon measure. Then, there exist a countable subset $H$ of $\mathbf{IV}_m(U)$ and a function $\xi: H \to \mathscr{P}$ such that $W$ is nonzero and integrally indecomposable whenever $W \in H$, and such that
    \begin{gather*}
        V(k) = \sum_{W \in H} \xi(W) W(k) \quad \te{whenever $k \in \mathscr{K}(U \times \mathbf{G}(n, m))$}, \\
        \te{and} \quad \|\updelta V\|(f) = \sum_{W \in H} \xi(W) \|\updelta W\|(f) \quad \te{whenever $f \in \mathscr{K}(U)$}.
    \end{gather*}
\end{intro_theorem}


\textbf{Organization of this paper.}
In Section \ref{section:notation}, we introduce the notation. 
In Section \ref{section:topology}, to extend the main theorem for non-integral varifolds, we recall the strong topology on the class of all Daniell integrals.
In Section \ref{section:integral indecomposability}, the integral indecomposability of integral varifolds is defined and its relation to the indecomposability employed in \cite[6.2]{Menne16a_MR3528825} is established.
In Section \ref{section:integral decomposition}, we prove the existence theorem of integral decomposition of varifolds.

\textbf{Acknowledgements.}
The author would like to thank his PhD advisor Prof. Ulrich Menne, Dr. Nicolau Sarquis Aiex, and Mr. Yu-Tong Liu for suggestions and consultations. The author was supported by NTNU ``Scholarship Pilot Program of the Ministry of Science and Technology to Subsidize Colleges and Universities in the Cultivation of Outstanding Doctoral Students''.

\section{Notation}
\label{section:notation}

\paragraph{Less common symbols and terminology.}
The difference of sets $A$ and $B$ is denoted by $A \without B$. 
The domain and image of a function $f$ are denoted by $\domain f$ and $\im f$. 
The topological boundary, closure and interior of a set $A$ are denoted by $\boundary A$, $\closure A$, and $\interior A$, respectively. 
If $X$ is a metric space with metric $\rho$, $A \subset X$, and $x \in X$, then the distance of $x$ to $A$ is defined by $\dist(x, A) = \inf \{ \rho(x, a) \with a \in A \}$.
Inner products are denoted by a $\bullet$ (see \cite[1.7.1]{Federer69_MR0257325}). 
The set of positive integer is denoted by $\mathscr{P}$ (see \cite[2.2.6]{Federer69_MR0257325}). 
The open and closed balls with center $a$ and radius $r$ are denoted by $\mathbf{U}(a, r)$ and $\mathbf{B}(a, r)$ (see \cite[2.8.1]{Federer69_MR0257325}). 
For integration, the alternate notations $\int f \ud \mu$, $\int f(x) \ud \mu \, x$ and $\mu(f)$ are employed; in this respect, $\mu$ integrability of $f$ means that $\int f \ud \mu$ is defined in $\overline{\mathbf{R}}$ and $\mu$ summability of $f$ means that $\int f \ud \mu \in \mathbf{R}$ (see \cite[2.4.2]{Federer69_MR0257325}). 
Whenever $A$ is a set and $f$ is an $\overline{\mathbf{R}}$-valued function, $\sum_{x \in A} f(x)$ denotes the numerical sum of $f$ over $A$ (see \cite[2.1.1]{Federer69_MR0257325}) and in case $A = \domain f$, we abbreviate $\sum f = \sum_{x \in A} f(x)$; moreover, if $G$ is a set of such functions $f$, we define $(\sum_{f \in G} f)(x) = \sum_{f \in G} f(x)$ whenever the right hand side is defined for $x$.
Whenever $X$ is a locally compact Hausdorff space, $\mathscr{K}(X)$ denotes the vector space of continuous real-valued functions on $X$ with compact support, and $\mathscr{K}(X)^*$ denotes the vector space of all Daniell integrals on $\mathscr{K}(X)$ (see \cite[2.5.14, 2.5.19]{Federer69_MR0257325}). 
The topology of $\mathscr{K}(X)$ is defined such that $\mathscr{K}(X)^*$ becomes the dual topological vector space to $\mathscr{K}(X)$ (see \cite[2.10, 2.11]{Menne16a_MR3528825}). 
Functions that are $k$ times continuously differentiable and submanifolds defined by such functions are termed ``of class $k$'' (see \cite[3.1.11, 3.1.19]{Federer69_MR0257325}).

Suppose $X$ is a metric space, $\phi$ and $\psi$ are Borel regular measures on $X$ such that every bounded subset of $X$ has finite measure, and we define a Borel regular measure by
\[
    \psi_{\phi}(A) = \inf \{\psi(B) \with \te{$B$ is a Borel set and $\phi(A \without B) = 0$}\}
\]
whenever $A \subset X$. In case $\psi_{\phi} = \psi$, we say $\psi$ is absolutely continuous with respect to $\phi$ (see \cite[2.9.1, 2.9.2]{Federer69_MR0257325}).
In general, the definition of $\psi_\phi$ is also employed while $X$ has an exhaustion by open sets on which $\psi$ and $\phi$ has finite measures.

Whenever $m \in \mathscr{P}$, $\mu$ measures an open subset $U$ of a normed space $X$, $a \in U$, and $\iota \with U \to X$ is the inclusion map, we abbreviate $\tangent^m(\iota_{\#} \mu, a)$ as $\tangent^m( \mu, x)$. 
Whenever $Y$ is a separable Banach space, $\mathscr{D}(U, Y)$ denotes the vector space of  $Y$-valued functions of class $\infty$ with compact support, of which the topology is defined as in \cite[2.13]{Menne16a_MR3528825} and $\mathscr{D}'(U, Y)$ denotes the dual topological vector space to $\mathscr{D}(U, Y)$.
For $T \in \mathscr{D}'(U, Y)$, $\|T\|$ is defined to be the largest Borel regular measure over $U$ such that
\[ \|T\|(G) = \sup \{ T(g) \with g \in \mathscr{D}(U, Y), \spt g \subset G, |g| \leq 1 \} \] 
whenever $G$ is an open subset of $U$ (see \cite[2.18]{Menne16a_MR3528825}); in case $\|T\|$ is a Radon measure, this concept agrees with \cite[4.1.5]{Federer69_MR0257325}, hence we say $T$ is \textit{representable by integration} and $T(g)$ continues to denote the value of the unique $\|T\|_{(1)}$ continuous extension of $T$ to $\mathbf{L}_1(\|T\|, Y)$ at $g \in \mathbf{L}_1(\|T\|, Y)$ and for every $\|T\|$ measurable set $A$, we define $T \restrict A \in \mathscr{D}'(U, Y)$ by $(T \restrict A) (g) = T(g_A)$, where $g_A(x) = g(x)$ for $x \in A$ and $g_A(x) = 0$ for $x \in U \without A$ (see \cite[2.18, 2.19, 2.20]{Menne16a_MR3528825}).

Whenever $n \in \mathscr{P}$, $U$ is an open subset of $\mathbf{R}^n$, and $m$ is a non-negative integer, the space of varifolds, rectifiable varifolds, and integral varifolds of dimension $m$ are denoted by $\mathbf{V}_m(U)$, $\mathbf{RV}_m(U)$ and $\mathbf{IV}_m(U)$, respectively (see \cite[3.1, 3.5]{Allard72_MR307015}); whenever $V \in \mathbf{V}_m(U)$ such that $\|\updelta V\|$ is a Radon measure, there exists an $\mathbf{R}^n$-valued function $\upeta(V, \cdot)$ defined by the requirement that, for $x \in U$,
\[ \upeta(V, x) \bullet u = \lim_{r \to 0+} \frac{ \updelta V (b_{x, r} \cdot u ) }{ \|\updelta V\| \mathbf{B}(x, r) } \quad \te{whenever $u \in \mathbf{R}^n$}, \]
where $b_{x, r}$ is the characteristic function of $\mathbf{B}(x, r)$ on $U$; hence $x \in \domain \upeta(V, \cdot)$ if and only if the above limit exists. This definition adapts \cite[4.3]{Allard72_MR307015} in the spirit of \cite[4.1.5]{Federer69_MR0257325}; in particular, $\upeta(V, \cdot)$ is $\|\updelta V\|$ measurable, $|\upeta(V, \cdot)| = 1$ for $\|\updelta V\|$ almost all $x$, and
\[ \updelta V(g) = \textstyle\int \upeta(V, x) \bullet g(x) \ud \|\updelta V\| \, x \quad \te{for $g \in \mathbf{L}_1(\|\updelta V\|, \mathbf{R}^n)$}. \]
Similarly, we also define a $\|V\|$ measurable $\mathbf{R}^n$-valued function $\mathbf{h}(V, \cdot)$ by the requirement that, for $x \in U$,
\[ \mathbf{h}(V, x) \bullet u = - \lim_{r \to 0+} \frac{ \updelta V (b_{x, r} \cdot u ) }{ \|V\| \mathbf{B}(x, r) } \quad \te{whenever $u \in \mathbf{R}^n$} \]
which satisfies
\[ \updelta V(g) = - \textstyle\int \mathbf{h}(V, x) \bullet g(x) \ud \|V\| \, x + \textstyle\int \upeta(V, x) \bullet g(x) \ud (\|\updelta V\| - \|\updelta V\|_{\|V\|}) \, x\]
whenever $g \in \mathbf{L}_1(\|\updelta V\|, \mathbf{R}^n)$.
If $E$ is $\|V\| + \|\updelta V\|$ measurable, then the distributional $V$ boundary of $E$ is given by 
\[ V \partial E = (\updelta V) \restrict E - \updelta(V \restrict E \times \mathbf{G}(n, m)) \in \mathscr{D}'(U, \mathbf{R}^n) \]
(see \cite[5.1]{Menne16a_MR3528825}).

\section{Topology}
\label{section:topology}

In this section, we present the necessary results about the strong topology.

\begin{definition}[see \protect{\cite[III.14, Example 4]{Bourbaki_TVS_MR910295}}]
    Suppose $X$ is a locally compact Hausdorff space. There exists a unique locally convex topology on $\mathscr{K}(X)^*$ termed \textit{strong topology} such that the sets
    \begin{equation*}
        \mathscr{K}(X)^* \cap \{ \mu \with \te{$|\mu(f)| < r $ for all $f \in B$} \}
    \end{equation*}
    corresponding to $r \in \mathbf{R}$, $r>0$ and bounded subset $B$ of $\mathscr{K}(X)$ give a
    local base at $0$.
\end{definition}


\begin{remark}[see \protect{\cite[III.12, Examples 1, 3]{Bourbaki_TVS_MR910295}} and \protect{\cite[III.23, Corollary 1]{Bourbaki_TVS_MR910295}}]
    The space $\mathscr{K}(X)^*$ equipped with the strong topology is complete.
\end{remark}


\begin{remark}[see \protect{\cite[2.11, 2.12]{Menne16a_MR3528825}} and \protect{\cite[III.5, Proposition 6]{Bourbaki_TVS_MR910295}}]
    Suppose $X$ has a sequence of compact subsets $K_i$ such that $K_i \subset \interior K_{i+1}$ for $i \in \mathscr{P}$ and $X = \bigcup_{i=1}^{\infty}K_i$. Then, the strong topology on $\mathscr{K}(X)^*$ is metrizable; in fact, it is generated by the family of semi-norms, with value
    \begin{equation*}
        \sup \{ \mu(f) \with f \in \mathscr{K}(X), \spt f \subset K_i, |f| \leq 1 \}
    \end{equation*}
    at $\mu \in \mathscr{K}(X)^*$, corresponding to $i \in \mathscr{P}$.
\end{remark}

\section{Integral indecomposability}
\label{section:integral indecomposability}

\begin{definition}
    \label{definition:integral indecomposability}
    Suppose $m, n \in \mathscr{P}$, $m \leq n$, $U$ is an open subset of $\mathbf{R}^n$, $V \in \mathbf{IV}_m(U)$ and $\|\updelta V\|$ is a Radon measure. Then $V$ is called \textit{integrally indecomposable} if and only if there exists no $W \in \mathbf{IV}_m(U)$ such that $W \leq V$, $W \neq 0$, $V-W \neq 0$, $\|V\| = \|W\| + \|V-W\|$ and $\| \updelta V \| = \| \updelta W \| + \| \updelta (V-W) \|$.
\end{definition}


The following basic results on distributions will be convenient for a later computation.


\begin{lemma}
    \label{lemma:extension of Riezs theorem}
    Suppose $n \in \mathscr{P}$, $U$ is an open subset of $\mathbf{R}^n$, $\mu$ is a Radon measure on $U$, $\eta$ is a $\mu$ measurable function with values in $\mathbf{R}^n$
    such that
    \begin{align*}
        { \textstyle\int_K |\eta| \ud \mu }<\infty \quad \te{for every compact subset $K$ of $U$}
    \end{align*}
    and $\Omega$ is the space of  continuous $\mathbf{R}^n$-valued functions with compact support.
    Then, for every open subset $G$ of $U$, we have
    \begin{equation*}
        { \textstyle\int_G |\eta| \ud \mu } = \sup \{ { \textstyle\int \omega(x) \bullet \eta(x) \ud \mu \, x } \with \omega\in\Omega, \spt\omega\subset G, |\omega|\leq 1 \}.
    \end{equation*}
\end{lemma}
\begin{proof}
    Define $T \in \mathscr{D}'(U, \mathbf{R}^n)$ by 
    \[
        T(\omega) = \textstyle\int \omega(x) \bullet \eta(x) \ud \mu \, x
    \]
    whenever $\omega \in \mathscr{D}(U, \mathbf{R}^n)$. Then, $T$ is representable by integration and
    \[
        \|T\|(f) \leq \sup \im |\eta| \mu(f) \quad \te{whenever $f \in \mathscr{K}(X)^+$},
    \]
    hence $\|T\| = \|T\|_{\mu}$. Therefore, it follows from \cite[3.2]{https://doi.org/10.48550/arxiv.2209.05955} that $\|T\| = \phi \restrict |\eta|$.
\end{proof}


\begin{remark}
    \label{remark:interchange restriction and total variation}
    Suppose $m, n \in \mathscr{P}$, $U$ is an open subset of $\mathbf{R}^n$, $T\in\mathscr{D}'(U,\mathbf{R}^m)$ is representable by integration and $E$ is $\|T\|$ measurable. Applying \ref{lemma:extension of Riezs theorem} with $\mu = \|T\| \restrict E$ and $\eta = \xi$ where $\xi$ is as in \cite[4.1.5]{Federer69_MR0257325}, then 
    \begin{equation*}
        \| T \restrict E \| = \|T\| \restrict E.
    \end{equation*}
\end{remark}


The rest of this section contributes to the relation between the integral indecomposability and the distributional boundary of sets.


\begin{lemma}
    \label{lemma:equivalence if}
    Suppose that $m, n \in \mathscr{P}$ with $m \leq n$, $U$ is an open subset of $\mathbf{R}^n$, $V \in \mathbf{V}_m(U)$, $\| \updelta V \|$ is a Radon measure, $E$ is $\| V \| + \| \updelta V \|$ measurable with $V \partial E = 0$, and $W = V \restrict E \times \mathbf{G}(n,m)$. Then, we have  
    \[
        \| \updelta V \| = \| \updelta W \| + \| \updelta (V-W) \|.
    \]
\end{lemma}
\begin{proof}
    We recall from \cite[5.3]{Menne16a_MR3528825} that 
    \begin{equation*}
        V \partial (U \without E) = - V \partial E = 0,
    \end{equation*}
    hence
    \begin{equation*}
        \updelta W = (\updelta V) \restrict E \quad \te{and} \quad \updelta (V-W) = (\updelta V) \restrict (U \without E).
    \end{equation*}
    Therefore, we have
    \begin{align*}
        \| \updelta V \| 
        &= \| \updelta V \| \restrict E + \| \updelta V \| \restrict (U \without E) \\
        &= \| (\updelta V) \restrict E \| + \| (\updelta V) \restrict (U \without E) \| \\
        &= \| \updelta W \| + \| \updelta (V-W) \|
    \end{align*}
    by the $\| \updelta V \|$ measurability of $E$ and \ref{remark:interchange restriction and total variation}.
\end{proof}


\begin{remark}
    In contrast to \ref{lemma:equivalence if}, it may happen that $W = V \restrict E \times \mathbf{G}(n,m)$ satisfies $\| \updelta V \| = \| \updelta W \| + \| \updelta (V-W) \|$, but $V \partial E \neq 0$. Let
    \begin{gather*}
        R_1 = \{ t(1, 0) \with 0 \leq t \in \mathbf{R} \}, \\
        R_2 = \{ t(\cos \theta, \sin \theta) \with 0 \leq t \leq 2 , t \in \mathbf{R} \}, \\
        R_3 = \{ t(\cos 2\theta, \sin 2\theta) \with 0 \leq t \leq 2, t \in \mathbf{R} \}, \\
        L = \{ (-1, t) \with t \in \mathbf{R} \},
    \end{gather*}
    where $0 < \theta < \pi$ with $\cos \theta = -1/2$ and we can define $V \in \mathbf{IV}_1(\mathbf{R}^2)$ satisfying the following conditions
    \begin{enumerate}
        \item 
            $\|V\| = \mathscr{H}^1 \restrict (R_1 \cup R_2 \cup R_3 \cup L)$.
        \item
            $\|\updelta V\|$ is a Radon measure.
        \item
            $\spt \|\updelta V\| = \{ (-1, 2), (-1, -2) \}$.
    \end{enumerate}
    Let $E = \mathbf{R}^2 \cap \{ (x, y) \with x > -1 \}$. Then, we have $\|\updelta W\| = \|\updelta V\|$ and $\|\updelta (V - W)\|$; however, $(\updelta V) \restrict E = 0$, hence $V \partial E = - \updelta W \neq 0$.
\end{remark}


Next, we will show that if $V \in \mathbf{IV}_m(U)$ and $\|\updelta V\| = \|\updelta V\|_{\|V\|}$, then the converse of \ref{lemma:equivalence if} holds.


\begin{theorem}
    \label{theorem:mean curvature indentity}
    Suppose that $m, n \in \mathscr{P}$, $m \leq n$, $U$ is an open subset of $\mathbf{R}^n$, $V, W \in \mathbf{IV}_m(U)$, $\|\updelta V\| + \|\updelta W\|$ is a Radon measure and 
    \[ A = \{ x \with \te{$\mathbf{\Theta}^m(\|V\|, x) > 0$ and $\mathbf{\Theta}^m(\|W\|, x) > 0$} \}. \]
    Then,
    \begin{equation*}
        \mathbf{h}(V, x) = \mathbf{h}(W, x) \quad \te{for $\mathscr{H}^m$ almost all $x$ in $A$}.
    \end{equation*}
\end{theorem}
\begin{proof}
    In view of \cite[3.5(1b)]{Allard72_MR307015}, the theorem \cite[4.8]{Menne13_MR3023856} reduces the problem to the case of submanifolds of class $2$.
\end{proof}


\begin{remark}
    In case $\|\updelta V\|$ is absolutely continuous with respect to $\|V\|$, we have $V$ is integrally indecomposable if and only if there exists no $W \in \mathbf{IV}_m(U)$ such that $W \leq V$, $W \neq 0$, $V - W \neq 0$, and
    \begin{center}
        $\|\updelta W\|$ is absolutely continuous with respect to $\|W\|$, \\
        $\|\updelta (V - W)\|$ is absolutely continuous with respect to $\|V - W\|$.
    \end{center}
    In fact, from \ref{theorem:mean curvature indentity}, we always have
    \[
        \|\updelta V\|_{\|V\|} = \|\updelta W\|_{\|W\|} + \|\updelta (V-W)\|_{\|V-W\|},
    \]
    and the assertion follows. Therefore, the present definition of indecomposability extends \cite[2.15]{Mondino14_MR3148123} when $\spt \|V\|$ is compact and $\|\updelta V\|$ is absolutely continuous with respect to $\|V\|$.
\end{remark}


\begin{corollary}
    \label{lemma:equivalence only if}
    Suppose that $m, n \in \mathscr{P}$, $m \leq n$, $U$ is an open subset of $\mathbf{R}^n$, $V \in \mathbf{IV}_m(U)$, $\| \updelta V \|$ is a Radon measure, $\|\updelta V\|$ is absolutely continuous with respect to $\|V\|$, $E$ is $\| V \| + \| \updelta V \|$ measurable and $W = V \restrict E \times \mathbf{G}(n,m)$ satisfies $\| \updelta V \| = \| \updelta W \| + \| \updelta (V-W) \|$. Then, $V \partial E = 0$.
\end{corollary}
\begin{proof}
    In view of \ref{theorem:mean curvature indentity}, we have
    \begin{equation*}
        \|\updelta V\|_{\|V\|} = \|\updelta W\|_{\|W\|} + \|\updelta (V-W)\|_{\|V-W\|},
    \end{equation*}
    and hence
    \begin{equation*}
        \| \updelta W \| - \|\updelta W\|_{\|W\|} = 0 = \| \updelta (V-W) \| - \|\updelta (V-W)\|_{\|V-W\|}.
    \end{equation*}
    Therefore, whenever $g \in \mathscr{D}(U, \mathbf{R}^n)$, we have
    \begin{align*}
        ((\updelta V) \restrict E)(g) 
        &= -{\textstyle\int_U \mathbf{h}(V, x) \bullet g(x) \ud \|W\| \, x}\\
        &= -{\textstyle\int_U \mathbf{h}(W, x) \bullet g(x) \ud \|W\| \, x}\\
        &= \updelta W(g)
    \end{align*}
    This shows $V \partial E = 0$.
\end{proof}

\section{Integral decomposition}
\label{section:integral decomposition}

\begin{definition}
    \label{definition:appropriate}
    Suppose $m, n \in \mathscr{P}$, $m \leq n$, $U$ is an open subset of $\mathbf{R}^n$ and $P \subset \mathbf{RV}_m(U)$. Then $P$ is called \textit{appropriate} if and only if
    \begin{enumerate}
        \item
            If $V, W \in P$, then $V+W \in P$.
        \item
            If $V \in P$, then $\mathbf{\Theta}^m(\|V\|, x) \geq 1$ for $\|V\|$ almost all $x$.
        \item 
            $P$ is closed with respect to the strong topology.
    \end{enumerate}
\end{definition}


Now, we aim to provide examples of appropriate classes; for this purpose, the following lemma is a powerful tool to verify the closedness of a class with respect to the strong topology.


\begin{lemma}
    \label{lemma:absolute continuity}
    Suppose $X$ is a locally compact Hausdorff space which is a countable union of compact subsets, $\mu_i \in \mathscr{K}(X)^*$ is monotone for each $i \in \mathscr{P}$, and $\mu_i \to \mu \in \mathscr{K}(X)^*$ with respect to the strong topology. Then
    \begin{equation*}
        \lim_{i \to \infty} \textstyle\int f \ud \mu_i = \textstyle\int f \ud \mu
    \end{equation*}
    whenever $f : X \to \mathbf{R}$ is bounded, the set $X \cap \{ x \with f(x) \neq 0 \}$ is contained in a compact subset of $U$, and $f$ is $\mu_i + \mu$ measurable for $i \in \mathscr{P}$, where $\mu_i$ and $\mu$ are identified with the associated $\mathscr{K}(X)$ regular measures. Moreover, if $\nu \in \mathscr{K}(X)^*$ is monotone and $\mu_i$ is absolutely continuous with respect to $\nu$ for $i \in \mathscr{P}$, then $\mu$ is absolutely continuous with respect to $\nu$.
\end{lemma}
\begin{proof}
    We abbreviate $|\nu| = \nu^+ + \nu^-$ for $\nu \in \mathscr{K}(X)^*$. In view of \cite[2.5.6]{Federer69_MR0257325}, for $i \in \mathscr{P}$, we have $\mathbf{L}_1(\mu + \mu_i) \subset \mathbf{L}_1(|\mu - \mu_i|)$ and
    \[ 
        \textstyle\int g \ud (\mu - \mu_i)^+ - \textstyle\int g \ud (\mu - \mu_i)^- = \mu(g) - \mu_i(g) \quad \te{for $g \in \mathbf{L}_1(\mu + \mu_i)$}. 
    \]
    In particular, we see
    \[ 
        \mu(f) - \mu_i(f) = \textstyle\int f \ud (\mu - \mu_i)^+ - \textstyle\int f \ud (\mu - \mu_i)^- \leq |\mu - \mu_i|(|f|) 
    \]
    and the main assertion follows from the strong convergence of $\mu_i$. For the postscript, suppose $A \subset X$ satisfies $\nu(A) = 0$ and choose by \cite[2.5.3]{Federer69_MR0257325} a $\mu + \mu_i$ measurable set $B$ such that $A \subset B$ and $\mu_i(B) = 0$ for $i \in \mathscr{P}$. From the main assertion, we infer $\mu(B) = 0$.
\end{proof}


\begin{lemma}
    \label{lemma:closedness}
    Suppose $m, n \in \mathscr{P}$, $m \leq n$ and $U$ is an open subset of $\mathbf{R}^n$. Then $\mathbf{RV}_m(U)$ and $\mathbf{IV}_m(U)$ are closed subsets with respect to the strong topology on $\mathbf{V}_m(U)$.
\end{lemma}
\begin{proof}
    Suppose $V_i \in \mathbf{RV}_m(U)$, $V \in \mathbf{V}_m(U)$, $V_i \to V$ with respect to the strong topology. Note that $\|V_i\| \to \|V\|$ with respect to the strong topology because $|p_{\#}(V - V_i)| \leq p_{\#}|V- V_i|$, where $p: U \times \mathbf{G}(n, m) \to U$ is the projection map and $|\nu| = \nu^+ + \nu^-$ whenever $\nu$ is a Daniell integral. From \cite[3.5(2)]{Allard72_MR307015}, we define a rectifiable varifold by
    \[
        W = \sum_{i=1}^{\infty}2^{-i}V_i \in \mathbf{RV}_m(U)
    \]
    with respect to the weak topology and for $i \in \mathscr{P}$, 
    \[
        \tangent^m(\|V_i\|, x) = \tangent^m(\|W\|, x) \quad \te{for $\|V_i\|$ almost all $x$},
    \]
    hence by \ref{lemma:absolute continuity},
    \begin{equation*}
        V(f) = \lim_{i \to \infty} \textstyle\int f(x, \tangent^m(\|V_i\|, x)) \ud \|V_i\| \, x = \textstyle\int f(x, \tangent^m(\|W\|, x)) \ud \|V\| \, x
    \end{equation*}
    whenever $f \in \mathscr{K}(U \times \mathbf{G}(n, m))$. Since $\|V\|$ is a Radon measure, we have
    \begin{equation*}
        \mathbf{\Theta}^{*m}(\|V\|, x) < \infty \quad \te{for $\mathscr{H}^m$ almost all $x$}
    \end{equation*}
    by \cite[2.10.19(3)]{Federer69_MR0257325}; thus, $\{ x \with \mathbf{\Theta}^{*m}(\|V\|, x) = \infty \}$ is a Borel set satisfying $\|V_i\| \{ x \with \mathbf{\Theta}^{*m}(\|V\|, x) = \infty \} = 0$ for $i \in \mathscr{P}$. In view of \ref{lemma:absolute continuity} and \cite[3.25]{https://doi.org/10.48550/arxiv.2206.14046}, we have $\|V\| \{ x \with \mathbf{\Theta}^{*m}(\|V\|, x) = \infty \} = 0$ and  $\|V\|$ is absolutely continuous with respect to $\|W\|$, hence $\|V\|$ is the weight of some member of $\mathbf{RV}_m(U)$ and $\tangent^m(\|V\|, x) = \tangent^m(\|W\|, x)$ for $\|V\|$ almost all $x$. Thus, 
    \begin{equation*}
        V(f) = \textstyle\int f(x, \tangent^m(\|V\|, x)) \ud \|V\| \, x
    \end{equation*}
    whenever $f \in \mathscr{K}(U \times \mathbf{G}(n, m))$.
    
    Finally, for every open subset $G$ of $U$ such that $\closure G$ is a compact subset of $U$, applying \ref{lemma:extension of Riezs theorem} with $\mu = \mathscr{H}^m \restrict \{ x \with \mathbf{\Theta}^m(\|V\|, x) + \mathbf{\Theta}^m(\|V_i\|, x) > 0 \}$ and $\eta = \mathbf{\Theta}^m(\|V\|, \cdot) - \mathbf{\Theta}^m(\|V_i\|, \cdot)$, we have
    \begin{align*}
        \textstyle\int_G |\mathbf{\Theta}^m(\|V\|, x) &- \mathbf{\Theta}^m(\|V_i\|, x)| \ud \mathscr{H}^m \, x \\
        &= \sup \{ \|V\|(g) - \|V_i\|(g) \with g \in \mathscr{K}(U), \spt g \subset G, |g| \leq 1 \}
    \end{align*}
    tends to $0$ as $i \to \infty$ because of the strong convergence of $\|V_i\|$. Therefore, $\mathbf{\Theta}^m(\|V\|, x)$ is the limit of a subsequence of $\mathbf{\Theta}^m(\|V_i\|, x)$ for $\mathscr{H}^m$ almost all $x$. In particular, if $V_i$ is integral for $i \in \mathscr{P}$, then so is $V$.
\end{proof}


\begin{example}
    Suppose $C$ is a closed subset of $\mathbf{R}$ satisfying $\inf C \geq 1$ and 
    \[
        c + d \in C \quad \te{whenever $c , d \in C$}.
    \]
    In view of \ref{lemma:closedness}, the class
    \[ 
        \mathbf{RV}_m(U) \cap \{ V \with \te{$\mathbf{\Theta}^m(\|V\|, x) \in C$ for $\|V\|$ almost all $x$} \} 
    \]
    is appropriate. For instance, such $C$ could be $\mathscr{P}$, $\mathbf{R} \cap \{ t \with 1 \leq t \}$, or
    \[
        \{ 1 \} \cup (\mathbf{R} \cap \{ t \with 2 \leq t \});
    \]
    for the last one, see also \cite{https://doi.org/10.48550/arxiv.2204.06491}.
\end{example}


\begin{definition}
    \label{definition:integral indecomposability with respect to appropriate classes}
    Suppose $m, n \in \mathscr{P}$, $m \leq n$, $U$ is an open subset of $\mathbf{R}^n$, $P$ is an appropriate subset of $\mathbf{RV}_m(U)$, $V \in P$ and $\|\updelta V\|$ is a Radon measure. Then $V$ is called \textit{indecomposable with respect to $P$} if and only if there exists no $W \in P \without \{0\}$ such that $V-W \in P \without \{0\}$ and $\| \updelta V \| = \| \updelta W \| + \| \updelta (V-W) \|$.
\end{definition}


\begin{definition}
    Suppose $m, n \in \mathscr{P}$, $m \leq n$, $U$ is an open subset of $\mathbf{R}^n$, $P$ is an appropriate subset of $\mathbf{RV}_m(U)$, $W, V \in P$ and $\|\updelta V\|$ is a Radon measure. Then $W$ is called a \textit{component of $V$ with respect to $P$} if and only if $W \neq 0$, $W \leq V$, $\|\updelta V\| = \|\updelta W\| + \|\updelta (V-W)\|$ and $W$ is indecomposable with respect to $P$.
\end{definition}


\begin{definition}
    \label{definition:decomposition}
    Suppose $m, n \in \mathscr{P}$, $m \leq n$, $U$ is an open subset of $\mathbf{R}^n$, $P$ is an appropriate subset of $\mathbf{RV}_m(U)$, $V \in P$ and $\|\updelta V\|$ is a Radon measure. Then a countable subfamily $H$ of $P$ together with a map $\xi \with H \to \mathscr{P}$ is called a \textit{decomposition of $V$ with respect to $P$} if and only if
    \begin{enumerate}
        \item 
            \label{definition:decomposition 1}
            Each member of $H$ is a component of $V$ with respect to $P$.
        \item
            \label{definition:decomposition 2}
            $\sum_{W \in H} \xi(W) W(k) = V(k)$ whenever $k \in \mathscr{K}(U \times \mathbf{G}(n, m))$.
        \item
            \label{definition:decomposition 3}
            $\sum_{W \in H} \xi(W) \| \updelta W \|(f) = \| \updelta V \|(f)$ whenever $f \in \mathscr{K}(U)$.
    \end{enumerate}
\end{definition}


\begin{example}
    Let $0 < \theta < \pi$ be such that $\cos \theta = 1/4$. Consider the six rays
    \begin{align*}
        R_1 &= \{ t (1, 0) \with 0 < t < \infty \}, \\
        R_2 &= \{ t (\cos \theta, \sin \theta) \with 0 < t < \infty \}, \\
        R_3 &= \{ t (\cos (\pi - \theta), \sin \theta) \with 0 < t < \infty \}, \\
        R_4 &= \{ t (-1, 0) \with 0 < t < \infty \}, \\
        R_5 &= \{ t (\cos (\pi - \theta), -\sin \theta) \with 0 < t < \infty \}, \\
        R_6 &= \{ t (\cos \theta, -\sin \theta) \with 0 < t < \infty \},
    \end{align*}
    in $\mathbf{R}^2$ and the associated varifolds $V_i \in \mathbf{IV}_1(\mathbf{R}^n)$ with $\|V_i\| = \mathscr{H}^1 \restrict R_i$. Note that the integral varifold defined by
    \[
        V = 2 (V_1 + V_2 + V_3 + V_4 + V_5 + V_6)
    \]
    is stationary. Let 
    \begin{gather*}
        H_1 = \{ V_1 + V_4, V_2 + V_5, V_3 + V_6 \}, \\
        H_2 = \{ V_1 + 2(V_3 + V_5), V_4 + 2(V_2 + V_6), V_1 + V_4 \}
    \end{gather*}
    and define $\xi_i: H_i \to \mathscr{P}$ for $i = 1, 2$ by
    \[
        \im \xi_1 = \{ 2 \}
         \quad \te{and} \quad \im \xi_2 = \{ 1 \}.
    \]
    Then, $(H_i, \xi_i)$ for $i = 1, 2$ are distinct decompositions of $V$ with respect to $\mathbf{IV}_1(\mathbf{R}^2)$. It shows that there may exist different types of decompositions for a varifold and the components need not to have constant density. Furthermore, the decompositions may fail to be unique even if $\|V\|$ has density $1$, see also \cite[6.13]{Menne16a_MR3528825}.
\end{example}


To prove the main theorem, the following a priori estimate is a key observation: under smallness conditions on the first variation, the weight measure on a ball has a positive lower bound. This will provide, locally, an upper bound of the number of varifolds in a decomposition; moreover, it also suggests a way to construct a decomposition.


\begin{lemma}[a priori estimate]
    \label{lemma:a priori estimate}
    Suppose $0<c<\infty$, $0<d<\infty$, $m, n \in \mathscr{P}$, $m \leq n$, $U$ is an open subset of $\mathbf{R}^n$, $a \in U$, $r > 0$, $\mathbf{B}(a, r) \subset U$, $V \in \mathbf{V}_m(U)$, $\|\updelta V\|$ is a Radon measure and
    \begin{gather*}
        \mathbf{\Theta}^{*m}(\|V\|, a) \geq d \\
        \|\updelta V\|\mathbf{B}(a, t) \leq c \boldsymbol{\alpha}(m)t^m \quad \te{for $0 < t < r$}.
    \end{gather*}
    Then, there holds
    \begin{equation*}
        \|V\| \mathbf{B}(a,r) \geq \boldsymbol{\alpha}(m)(d-c r)r^m.
    \end{equation*}
\end{lemma}
\begin{proof}
    From \cite[4.5, 4.6]{Menne16a_MR3528825}, we have
    \begin{align*}
        s^{-m} \|V\| \mathbf{B}(a, s) 
        &\leq r^{-m} \|V\| \mathbf{B}(a, r)\\
        & \phantom{\leq} \ + { \textstyle\int_s^r t^{-m-1} \int_{\mathbf{B}(a, t)} (x-a) \bullet \upeta(V, x) \ud \|\updelta V\| \, x \ud \mathscr{L}^1 \, t }
    \end{align*}
    whenever $0 < s \leq r$; note that the last term is less than
    \begin{align*}
        { \textstyle\int_s^r t^{-m} \|\updelta V\| \mathbf{B}(a,t) \ud \mathscr{L}^1 \, t } \leq c \boldsymbol{\alpha}(m)(r-s)
    \end{align*}
    hence, letting $s \to 0+$,
    \begin{align*}
        r^{-m} \|V\| \mathbf{B}(a, r) \geq \boldsymbol{\alpha}(m) \mathbf{\Theta}^{*m}(\|V\|, a) - c \boldsymbol{\alpha}(m)r \geq \boldsymbol{\alpha}(m)(d-c r)
    \end{align*}
    which means $\|V\| \mathbf{B}(a, r) \geq \boldsymbol{\alpha}(m)(d-c r)r^m$.
\end{proof}


\begin{definition}
Suppose $m, n \in \mathscr{P}$, $m \leq n$, $U$ is an open subset of $\mathbf{R}^n$, and $P$ is an appropriate subset of $\mathbf{RV}_m(U)$. Then, $\Xi$ denotes the class of all functions $\xi$ such that 
\begin{gather*}
    \te{$\domain \xi$ is a finite subset of $P \without \{0\}$}, \quad \im \xi \subset \mathscr{P},\\
    \sum_{W \in \domain \xi} \xi(W) \| \updelta W \| = \| \updelta \mathbf{v}(\xi) \|,
\end{gather*}
where the map $\mathbf{v}: \Xi \to P$ is defined by 
\[
    \mathbf{v}(\xi)(f) = \sum_{W \in \domain \xi} \xi(W) W(f) \quad \te{whenever $f \in \mathscr{K}(U \times \mathbf{G}(n, m))$}.
\]
Furthermore, $\xi \in \Xi$ is called \textit{maximal with respect to a Borel set $B$} if and only if $\|W\|(B) > 0$ for all $W \in \domain \xi$ and 
\[
    \sum \xi \geq \sum \rho
\]
whenever $\rho \in \Xi$ satisfies
\[
    \mathbf{v}(\rho) = \mathbf{v}(\xi)
\]
and $\|X\|(B) > 0$ for all $X \in \domain \rho$. 
We say \textit{$W$ splits $V$ in $P$} if and only if $W \in P$, $V \in P$, $V - W \in P$, and $\|\updelta W\| + \|\updelta (V - W)\| = \|\updelta V\|$.
\end{definition}


\begin{theorem}
    \label{theorem:integral decomposition}
    Suppose $m, n \in \mathscr{P}$, $m \leq n$, $U$ is an open subset of $\mathbf{R}^n$, $P \subset \mathbf{RV}_m(U)$ is appropriate, $V \in P$ and $\|\updelta V\|$ is a Radon measure. Then, there exists a decomposition of $V$ with respect to $P$.
\end{theorem}
\begin{proof}
    Assume $V \neq 0$. Define $\delta_i = \boldsymbol{\alpha}(m)2^{-m-1}i^{-m}$, $\varepsilon_i = 2^{-1}i^{-1}$ for $i \in \mathscr{P}$ and let $A_i$ denote the the Borel set of $a \in \mathbf{R}^n$ satisfying
    \begin{gather*}
        |a| \leq i, \quad \mathbf{U}(a, 2 \varepsilon_i) \subset U, \quad 1 \leq \mathbf{\Theta}^m(\|V\|, a) < \infty\\
        \| \updelta V \|\mathbf{B}(a, r) \leq \boldsymbol{\alpha}(m) i r^m \quad \te{for $0 < r < \varepsilon_i$}
    \end{gather*}
    whenever $i \in \mathscr{P}$. Clearly, $A_i \subset A_{i+1}$ for $i \in \mathscr{P}$ and $\|V\|(U \without \bigcup_{i = 1}^\infty A_i) = 0$ by \cite[3.5 (1a)]{Allard72_MR307015} and \cite[2.8.18, 2.9.5]{Federer69_MR0257325}. For each $i \in \mathscr{P}$, we infer from \ref{lemma:a priori estimate} that
    \begin{gather*}
        \|W\| \mathbf{B}(a, \varepsilon_i) \geq \delta_i \quad \te{whenever $a \in A_i$ and $W \in P$ satisfy} \\
        \te{$W \leq V$, $\|\updelta W\| \leq \|\updelta V\|$, and $\mathbf{\Theta}^{*m}(\|W\|, a) \geq 1$}
    \end{gather*}
    and hence
    \begin{align*}
        \delta_i \sum \xi
        &\leq \sum_{W \in \domain \xi} \xi(W) \|W\| \{ x \with \dist(x, A_i) \leq \varepsilon_i \} \\
        &= \|\mathbf{v}(\xi)\| \{ x \with \dist(x, A_i) \leq \varepsilon_i \} < \infty
    \end{align*}
    whenever $\xi \in \Xi$ satisfies $\|W\|(A_i) > 0$ for all $W \in \domain \xi$, $\mathbf{v}(\xi) \leq V$, and $\|\updelta \mathbf{v}(\xi)\| \leq \|\updelta V\|$.
    
    Since $V \neq 0$, there exists a $\lambda \in \mathscr{P}$ such that $\|V\|(A_{\lambda}) > 0$. From now on, we replace $A_i$ by $A_{i + \lambda}$ for $i \in \mathscr{P}$. Let
    \begin{gather*}
        R = P \cap \{ W \with W \leq V, \|\updelta W\| \leq \|\updelta V\| \}, \\
        P_i = R \cap \{ W \with \|W\|(A_i) > 0 \} \quad \te{whenever $i \in \mathscr{P}$}.
    \end{gather*}
    Then, we may select functions $c_i: P_i \to \Xi$ such that $\mathbf{v}(c_i(W)) = W$ and $c_i(W)$ is maximal with respect to $A_i$; in particular, $\domain c_i(W) \subset P_i$ whenever $W \in P_i$.
    
    From now on, we will use the convention that $\infty \cdot 0 = 0$. Let $\Sigma$ be the class of all sequences $Z_1, Z_2, Z_3, \dotsc$ in $P$ satisfying $Z_1 = V$ and $Z_{i+1} \in \domain c_i(Z_i)$ and abbreviate $\lim Z = \lim_{i \to \infty} Z_i$ for $Z \in \Sigma$. 
    Let $C = \Sigma \cap \{ Z \with \lim Z \neq 0 \}$ and define $\nu: \Sigma \to \mathscr{P} \cup \{\infty\}$ by
    \[
        \nu(Z) = \prod_{i=1}^\infty c_i(Z_i)(Z_{i+1}).
    \]
    Note that for $Z \in \Sigma$, and $i, j \in \mathscr{P}$ with $i \leq j$, we have 
    \[
        Z_i \geq \prod_{k=i}^j c_k(Z_k)(Z_{k+1})Z_{j+1},
    \]
    hence
    \[
        V \geq \prod_{k=1}^j c_k(Z_k)(Z_{k+1})Z_{j+1} \geq \nu(Z) \lim Z;
    \]
    thus, $\im (\nu | C) \subset \mathscr{P}$. 
    
    Now, we aim to prove $C$ is countable, 
    \[
        \sum_{Z \in C} \nu(Z) \lim Z \leq V, \quad \te{and} \quad \sum_{Z \in C} \nu(Z) \|\updelta \lim Z\| \leq \|\updelta V\|.
    \]
    For $i \in \mathscr{P}$, if $i = 1$, let $F_1 = \{ V \}$, and if $i \geq 2$, let $F_i$ be the set of all finite sequences $Z_1, Z_2, \dotsc, Z_i$ such that $Z_{j+1} \in \domain c_j(Z_j)$ for $1 \leq j \leq i-1$. We define, for $i \in \mathscr{P}$, the restriction map
    \[
        R_i: \Sigma \cup \bigcup_{i < j \in \mathscr{P}} F_j \to F_i, \quad R_i(Z) = Z| \mathscr{P} \cap \{ k \with 1 \leq k \leq i \}.
    \]
    Observe that whenever $i \in \mathscr{P}$, $Y \in F_i$, and $F$ is a finite subset of $\Sigma$ satisfying $R_i[F] = \{ Y \}$, there exists $j \in \mathscr{P}$ such that $j > i$ and $R_j|F$ is injective. Therefore, we have
    \begin{align}
        \label{inequality: comparison1-1}
        \begin{split}
        \sum_{Z \in F} \nu(Z) \lim Z 
        &\leq \sum_{Z \in F} \prod_{k=1}^{j-1} c_k(Z_k)(Z_{k+1})Z_j \\
        &\leq \sum_{X \in F_j, R_i(X) = Y} \prod_{k=1}^{j-1} c_k(X_k)(X_{k+1})X_j \\
        &\leq \prod_{k=1}^{i-1} c_k(Y_k)(Y_{k+1})Y_i.
        \end{split}
    \end{align}
    Similarly, we have
    \begin{equation}
        \label{inequality: comparison1-2}
        \sum_{Z \in F} \nu(Z) \|\updelta \lim Z\| \leq \prod_{k=1}^{i-1} c_k(Y_k)(Y_{k+1}) \|\updelta Y_i\|.
    \end{equation}
    Choosing countable dense subsets of $\mathscr{K}(U \times \mathbf{G}(n, m))^+$ and $\mathscr{K}(U)^+$, we conclude from \eqref{inequality: comparison1-1}, \eqref{inequality: comparison1-2} with $i = 1$ and \cite[2.1.1(3)]{Federer69_MR0257325} that $C$ is countable and
    \begin{gather}
        \label{inequality: comparison2}
        \begin{split}
            \sum_{Z \in \Sigma} \nu(Z) \lim Z &\leq V, \\
            \sum_{Z \in \Sigma} \nu(Z) \|\updelta \lim Z\| &\leq \|\updelta V\|.
        \end{split}
    \end{gather}
    
    Next, we will prove the equalities in \eqref{inequality: comparison2} hold. Let $B$ consist of all $x \in \bigcup_{i=1}^\infty A_i$ such that
    \begin{gather*}
        \mathbf{\Theta}^m(\|W\|, x) \in \{ 0 \} \cup ( \mathbf{R} \cap \{ t \with 1 \leq t \} ), \\
        \mathbf{\Theta}^m(\|\lim Z\|, x) = \lim_{i \to \infty} \mathbf{\Theta}^m(\|Z_i\|, x)
    \end{gather*}
    whenever $W \in C \cup \bigcup_{Z \in C} \im Z$ and $Z \in C$. Note that
    \[
        \mathscr{H}^m( \{ x \with \mathbf{\Theta}^m(\|V\|, x) > 0 \} \without B ) = 0
    \]
    by \cite[3.5(1b)]{Allard72_MR307015} since $P$ is appropriate and $C \cup \bigcup_{Z \in C} \im Z$ is countable. For $x \in B$ and $Z \in \Sigma$, we abbreviate $\Theta_Z(x) = \lim_{i \to \infty} \mathbf{\Theta}^m(\|Z_i\|, x)$. Let $x \in B$. Similarly as \eqref{inequality: comparison2}, we have
    \[
        \sum_{Z \in \Sigma} \nu(Z) \Theta_Z(x) \leq \mathbf{\Theta}^m(\|V\|, x) < \infty;
    \]
    in particular, $E_x = \Sigma \cap \{ Z \with \Theta_Z(x) > 0 \}$ is a finite subset of $C$ because either $\Theta_Z(x) = 0$ or $\Theta_Z(x) \geq 1$, and $\nu(Z) \geq1$ for $Z \in \Sigma$, and $B \subset \bigcup_{i=1}^\infty A_i$. Accordingly, there exists $i \in \mathscr{P}$ such that $R_j | E_x$ is injective for $i < j \in \mathscr{P}$. Observe that
    \[
        R_j[E_x] = F_j \cap \{ Z \with \mathbf{\Theta}^m(\|Z_j\|, x) > 0 \} \quad \te{whenever $i < j \in \mathscr{P}$},
    \]
    and that
    \[
        \sum_{Z \in F_j} \prod_{k = 1}^{j-1} c_k(Z_k)(Z_{k+1}) \mathbf{\Theta}^m(\|Z_j\|, x) = \mathbf{\Theta}^m(\|V\|, x) \quad \te{whenever $i < j \in \mathscr{P}$};
    \]
    letting $j \to \infty$, we conclude
    \begin{equation}
        \label{equation: density equation}
        \sum_{Z \in E_x} \nu(Z) \mathbf{\Theta}^m(\|\lim Z\|, x) = \mathbf{\Theta}^m(\|V\|, x).
    \end{equation}
    If $x \in U$ satisfies $\mathbf{\Theta}^m(\|V\|, x) = 0$, then \eqref{equation: density equation} is trivial. Therefore, \eqref{equation: density equation} holds for $\mathscr{H}^m$ almost all $x \in U$ and we have
    \[
        \sum_{Z \in C} \nu(Z) \|\lim Z\| = \|V\|;
    \]
    it follows that $\tangent^m(\|\lim Z\|, x) = \tangent^m(\|V\|, x)$ whenever $x \in B$ and $Z \in E_x$, hence
    \begin{gather*}
        \sum_{Z \in C} \nu(Z) \lim Z = V, \\
        \sum_{Z \in C} \nu(Z) \|\updelta \lim Z\| = \|\updelta V\|.
    \end{gather*}
    
    Now, we will prove that for $i \in \mathscr{P}$ and $Y \in F_i$, there holds
    \[
        \sum_{Z \in C, R_i(Z) = Y} \nu(Z) \lim Z = \prod_{k=1}^{i-1} c_k(Y_k)(Y_{k+1})Y_i,
    \]
    The case $i = 1$ is treated in the preceding paragraph. Suppose the equation holds for some $i \in \mathscr{P}$. From \eqref{inequality: comparison1-1} we have
    \begin{gather*}
        \sum_{Z \in C, R_{i+1}(Z) = X} \nu(Z) \lim Z \leq \prod_{k=1}^{i} c_k(X_k)(X_{k+1}) X_{i+1}, \\
        \sum_{W \in F_{i+1}, R_i(W) = Y} \prod_{k=1}^{i} c_k(W_k)(W_{k+1}) W_{i+1} \leq \prod_{k=1}^{i-1} c_k(Y_k)(Y_{k+1}) Y_{i}
    \end{gather*}
    whenever $X \in F_{i+1}$ and $Y \in F_i$. Thus, by \cite[2.1.1(9)]{Federer69_MR0257325}, we have
    \begin{align*}
        \prod_{k=1}^{i-1} c_k(Y_k)(Y_{k+1}) Y_{i}
        &=\sum_{Z \in C, R_i(Z) = Y} \nu(Z) \lim Z \\
        &= \sum_{X \in F_{i+1}, R_i(X) = Y} \sum_{Z \in C, R_{i+1}(Z) = X} \nu(Z) \lim Z \\
        &\leq \sum_{X \in F_{i+1}, R_i(X) = Y} \prod_{k=1}^{i} c_k(X_k)(X_{k+1}) X_{i+1} \\
        &\leq \prod_{k=1}^{i-1} c_k(Y_k)(Y_{k+1}) Y_{i}
    \end{align*}
    whenever $Y \in F_i$, which forces
    \[
        \sum_{Z \in C, R_{i+1}(Z) = X} \nu(Z) \lim Z = \prod_{k=1}^{i} c_k(X_k)(X_{k+1}) X_{i+1}
    \]
    whenever $X \in F_{i+1}$. Similarly, we also have
    \[
        \sum_{Z \in C, R_i(Z) = Y} \nu(Z) \|\updelta \lim Z\| = \prod_{k=1}^{i-1} c_k(Y_k)(Y_{k+1}) \|\updelta Y_i\|
    \]
    whenever $i \in \mathscr{P}$ and $Y \in F_i$; in particular, $\lim Z$ splits $Z_i$ in $P$ whenever $Z \in C$ and $i \in \mathscr{P}$.
    
    Finally, we will prove $\lim Z$ is indecomposable with respect to $P$ whenever $Z \in C$. If it were not the case, there exists $W \in P \without \{0\}$ such that $W \leq \lim Z$, $\lim Z - W \in P \without \{0\}$, and $\|\updelta \lim Z\| = \|\updelta W\| + \|\updelta (\lim Z - W)\|$. We may choose $i$ large enough such that $\|W\|(A_i) > 0$ and $\|\lim Z - W\|(A_i) > 0$, since $\|W\| + \|\lim Z - W\| \leq \|V\|$. Then, $W$ splits $Z_{i+1}$, $\|W\|(A_i) > 0$ and $\|Z_{i+1} - W\|(A_i) \geq \|\lim Z - W\|(A_i) > 0$, a contradiction to that $c_i(Z_i)$ is maximal with respect to $A_i$.
    
    Let $H = \{ \lim Z \with Z \in C \}$ and define $\xi(W) = \sum_{C \in Z, \lim Z = W} \nu(Z)$, then $(H, \xi)$ is a decomposition of $V$ with respect to $P$.
\end{proof}


\begin{remark}
    The structure of the proof of \ref{theorem:integral decomposition} is similar to the one of \cite[6.12]{Menne16a_MR3528825}. However, since we allow to decompose varifolds not just by restriction to subsets, we should take the multiplicity of varifolds into account which makes it much more complicated to verify the condition \ref{definition:decomposition}\eqref{definition:decomposition 2}.
\end{remark}


\medskip

\bibliographystyle{alpha}
\bibliography{ref}

\begin{thebibliography}{MS22b}

\bibitem[All72]{Allard72_MR307015}
William~K. Allard.
\newblock On the first variation of a varifold.
\newblock {\em Ann. of Math. (2)}, 95:417--491, 1972.

\bibitem[Alm00]{Almgren2000_MR1777737}
Frederick~J. Almgren, Jr.
\newblock {\em Almgren's big regularity paper}, volume~1 of {\em World
  Scientific Monograph Series in Mathematics}.
\newblock World Scientific Publishing Co., Inc., River Edge, NJ, 2000.
\newblock $Q$-valued functions minimizing Dirichlet's integral and the
  regularity of area-minimizing rectifiable currents up to codimension 2, With
  a preface by Jean E. Taylor and Vladimir Scheffer.

\bibitem[Bou87]{Bourbaki_TVS_MR910295}
N.~Bourbaki.
\newblock {\em Topological vector spaces. {C}hapters 1--5}.
\newblock Elements of Mathematics (Berlin). Springer-Verlag, Berlin, 1987.
\newblock Translated from the French by H. G. Eggleston and S. Madan.

\bibitem[DLS11]{DLS11_MR2663735}
Camillo De~Lellis and Emanuele~Nunzio Spadaro.
\newblock {$Q$}-valued functions revisited.
\newblock {\em Mem. Amer. Math. Soc.}, 211(991):vi+79, 2011.

\bibitem[Fed69]{Federer69_MR0257325}
Herbert Federer.
\newblock {\em Geometric measure theory}.
\newblock Die Grundlehren der mathematischen Wissenschaften, Band 153.
  Springer-Verlag New York, Inc., New York, 1969.

\bibitem[Men13]{Menne13_MR3023856}
Ulrich Menne.
\newblock Second order rectifiability of integral varifolds of locally bounded
  first variation.
\newblock {\em J. Geom. Anal.}, 23(2):709--763, 2013.

\bibitem[Men16]{Menne16a_MR3528825}
Ulrich Menne.
\newblock Weakly differentiable functions on varifolds.
\newblock {\em Indiana Univ. Math. J.}, 65(3):977--1088, 2016.

\bibitem[Mon14]{Mondino14_MR3148123}
Andrea Mondino.
\newblock Existence of integral {$m$}-varifolds minimizing {$\int|A|^p$} and
  {$\int|H|^p,\,p>m,$} in {R}iemannian manifolds.
\newblock {\em Calc. Var. Partial Differential Equations}, 49(1-2):431--470,
  2014.

\bibitem[MS22a]{https://doi.org/10.48550/arxiv.2206.14046}
Ulrich Menne and Christian Scharrer.
\newblock A priori bounds for geodesic diameter. {P}art {I}. {I}ntegral chains
  with coefficients in a complete normed commutative group.
\newblock \url{https://arxiv.org/abs/2206.14046}, 2022.

\bibitem[MS22b]{https://doi.org/10.48550/arxiv.2209.05955}
Ulrich Menne and Christian Scharrer.
\newblock A priori bounds for geodesic diameter. part {II}. {F}ine
  connectedness properties of varifolds.
\newblock \url{https://arxiv.org/abs/2209.05955}, 2022.

\bibitem[PS22]{https://doi.org/10.48550/arxiv.2204.06491}
Alessandro Pigati and Daniel Stern.
\newblock Quantization and non-quantization of energy for higher-dimensional
  {G}inzburg--{L}andau vortices.
\newblock \url{https://arxiv.org/abs/2204.06491}, 2022.

\end{thebibliography}

\medskip \noindent \textsc{Affiliation}

\noindent
Department of Mathematics \\
National Taiwan Normal University \\
No.88, Sec.4, Tingzhou Rd. \\
Wenshan Dist., \textsc{Taipei City 116059 \\
	Taiwan (R.O.C.)}

\medskip \noindent \textsc{Email address}

\medskip \noindent
\href{mailto:hsinchuangchou@gmail.com}{hsinchuangchou@gmail.com}

\end{document}